\documentclass{amsart}
\usepackage{cite, color, fancyhdr, lipsum}
\usepackage[foot]{amsaddr}
\usepackage{thmtools}
\usepackage{thm-restate}

\usepackage{hyperref}
\usepackage{cleveref}

\newcommand{\R}{{\mathbb R}}

\newcommand{\M}{{\mathcal M}_{n, k}}

\newcommand{\conv}{{\mathrm{conv}}}

\newtheorem{theorem}{Theorem}[section]
\newtheorem{lemma}[theorem]{Lemma}
\newtheorem{proposition}[theorem]{Proposition}
\newtheorem{cor}[theorem]{Corollary}
\newtheorem{defin}{Definition}

\DeclareMathOperator{\Cube}{Cube}
\DeclareMathOperator{\Coor}{Coor}
\DeclareMathOperator{\Cross}{Cross}

\begin{document}

\title[Cubes, Trees, and No $k$-equals]{On the topology of no $k$-equal spaces}
\author{Yuliy Baryshnikov}
\address[Yuliy Baryshnikov]{University of Illinois at Urbana-Champaign}
\author{Caroline Klivans}
\address[Caroline Klivans]{Brown University}
\author{Nicholas Kosar}
\address[Nicholas Kosar]{University of Illinois at Urbana-Champaign}
\email{kosar2@illinois.edu}

\begin{abstract}
  We consider the topology of real no $k$-equal spaces via the theory of cellular spanning trees.   Our main theorem proves that
the rank of the $(k-2)$-dimensional homology of the no $k$-equal subspace of $\R$ is equal to
  the number of facets in a $k$-dimensional spanning tree of the $k$-skeleton of the $n$-dimensional hypercube.
\end{abstract}

\maketitle

\section{Introduction}

For any topological space $X$, the $n^{th}$ no $k$-equal space of $X$
consists of the collection of all sets of $n$ points on $X$ such that
no $k$ of them are equal. The important special case $k=2$ yields the
configuration space of $X$. The study of no $k$-equal spaces for
general $k$ and $X = \R$ started with work in complexity theory by
Bj\"{o}rner, Lov\'{a}sz, and Yao \cite{BLY}.  Consider the following
problem: Given $n$ real numbers, determine if any $k$ of them are
equal.
 In \cite{BLY}, the authors sought to bound the depth of a linear
 decision tree for this problem.  In a novel application of algebraic
 combinatorics, the task was reposed as a subspace arrangement
 membership problem so that the complexity could be bounded by the
 Betti numbers of a topological space.

Bj\"{o}rner and Welker first determined the Betti numbers of the no
$k$-equal spaces of $\R$ \cite{BW}.  Their work used the techniques of
both Goresky-MacPherson~\cite{GM} and Ziegler-\v{Z}ivaljevi\'c~\cite{ZZ}
which provide methods to derive the topology of complements of
arrangements in terms of combinatorics of posets.  Further work by
various authors determined the cohomology rings and other properties
of no $k$-equal spaces; see e.g. \cite{Yuz, YB2, DT}.  In particular,
Baryshnikov and Dobrinskaya-Turchin gave explicit geometric
representatives for homology.

In another direction, for any $n$-dimensional cellular complex
$\Sigma$, and any $k \leq n$,  one can define a $k$-dimensional spanning
tree of $\Sigma$ as a certain subset of the $k$-skeleton of $\Sigma$. 
Cellular spanning trees capture the complexity of a space by generalizing
the well-known properties of spanning trees of graphs.  
 The notion of a higher dimensional spanning tree has its
origins in work of Bolker \cite{Bolk} and Kalai \cite{Kalai}.  More
recently, there has been much activity in developing the theory of
such trees; see e.g. \cite{DKM09, Lyons} and \cite{chapter} for an
overview of the topic.  Importantly, higher dimensional trees are
formulated algebraic and topologically with the graphical
requirements of a tree generalized in terms of homology and Betti numbers.

We connect these two areas of research in our main theorem which
equates the Betti number of the real no $k$-equal space, i.e. the complexity
bound of \cite{BLY} and \cite{BW}, with the size of a spanning tree of
the hypercube:

\begin{restatable}{thm}{goldbach}
  \label{main}
The rank of the $(k-2)$-dimensional homology group of the no $k$-equal subspace of $\R$ is equal to 
the number of facets in a $k$-dimensional spanning tree of the $k$-skeleton of the $n$-dimensional hypercube.
\end{restatable}

The numerical result of Theorem~\ref{main} can be noted independently
of any connection between the subspace arrangements and higher
dimensional spanning trees. Here, however, we offer a geometric
relationship between the two objects via an elementary construction we
call the {\em simplicial resolution}.  One could achieve the same
result using  homotopy colimits, but we specifically opted for a
more geometric route.  Hence,  we achieve the equality in
Theorem~\ref{main} without needing any knowledge of the explicit
values involved.

Theorem~\ref{main} should be seen in two ways.  First, it answers the
question of \emph{why} the Betti numbers of the real no $k$-equal
space are given by the sizes of trees of the cube.  Second, it is a
demonstration of a new approach to determining the topology of
complements of arrangements using combinatorial considerations but with no
need of poset analysis.

Additionally, we show a second situation where this idea may be used by
generalizing Theorem~\ref{main} to an arrangement that has not yet been
studied: the comb no $k$-equal arrangement.

\begin{restatable}{thm}{name}
\label{general}
The rank of the $(k-2)$-dimensional homology group of the generic comb no $k$-equal subspace of $\R$ is equal to 
the number of facets in a $k$-dimensional spanning tree of the $k$-skeleton of a pile of $n$-dimensional hypercubes.
\end{restatable}

In the following sections we introduce no $k$-equal spaces, higher dimensional trees, and simplicial resolutions.  In section \ref{Proof}, we prove our main result Theorem \ref{main}. Finally, in section \ref{piles}, we define the relevant notions and prove Theorem~\ref{general}.

\subsection{Acknowledgments}
 The work here began as a collaboration during the semester program
 Topology in Motion held at the Institute for Computational and
 Experimental Research in Mathematics (ICERM) in Fall 2016.  NK was
 supported in part by the NSF through the grant DMS-1622370.  The authors
 thank Eva-Maria Feichtner and Hannah Alpert for interesting
 discussions.

\section{No $k$-equal subspaces}

\begin{defin}
For a topological space $X$, the \emph{$n^{th}$ no $k$-equal space of $X$}
consists of the collection of all sets of $n$ points on $X$ such that no $k$ of them are 
equal.  
\end{defin}
The important special case $k=2$ yields the configuration space of $X$.
Here, we will consider arbitrary $k$ and $X = \R$.  This space was first
studied explicitly by Bj\"{o}rner, Lov\'{a}sz,  and Yao~\cite{BLY} in connection to complexity theory.  In order to understand the no $k$-equal subspace arrangement of $\R$, 
 it is easier to first consider the collections of points 
of the complement.
  This gives an arrangement of linear subspaces:

\begin{defin}
The \emph{$k$-equal arrangement}  $\mathcal{A}_{n, k}$ of $\R^n$ is the subspace arrangement consisting of all subspaces of the form $\{x_{i_1} = \cdots = x_{i_k}\}$ for $1 \leq i_1 < \ldots < i_k \leq n$.
\end{defin}

Again, $k=2$ yields an important special case.  The arrangement
$\mathcal{A}_{n,2}$ consists of all {hyperplanes} of the form $\{x_i = x_j\}$
for $ 1 \leq i < j \leq n$ and is known as the {Braid}
arrangement of $\R^n$ or equivalently the Coxeter arrangement of type
$A$ in $\R^n$.  For $k>2$, $\mathcal{A}_{n,k}$ consists not of hyperplanes but subspaces of codim-$(k-1)$.

\begin{defin}
The \emph{$n^{th}$ no $k$-equal space} $\M$ of $\R$ is the complement
in $\R^n$ of the $k$-equal arrangement, $$\M = \R^n \setminus \mathcal{A}_{n,k}.$$
\end{defin}

The space $\mathcal{M}_{n,2}$, which is the complement of the Braid
arrangement in $\R^n$, is simply a union of disjoint contractible cones.  The
closure of any one cone is a fundamental chamber for the type $A$ Weyl
group.  The complexified picture, $\mathcal{M}^{\mathbb{C}}_{n,2}$,
equal to the complement of $\mathcal{A}_{n,k}$ in $\mathbb{C}^n$ has
richer topology and is known as the pure Braid space.  For $k >2$, the
real picture also becomes non-trivial.

Bj\"{o}rner, Lov\'{a}sz, and Yao
 were investigating the following problem:
given $n$ real numbers, decide if any $k$ of them are equal.  In terms
of the spaces defined above, the problem becomes: given a point in
$\R^n$, decide if it lies on $\mathcal{A}_{n,k}$.  A main result of~\cite{BLY}
is that (in a certain formal sense) the complexity of answering this
question can be bounded by the $(k-2)^{\textrm{nd}}$ Betti number of $\M$.

Bj\"{o}rner and Welker were the first to explicitly compute these Betti numbers.  The results of \cite{BW} 
give much finer information than we restate here including an understanding of the topology in the complexified case.  We need only the following.  
\begin{theorem}[Theorem 1.1 of \cite{BW}]
The cohomology groups of $\mathcal{M}_{n,k}$ are free.  Furthermore, 
$$\textrm{rank } H^{k-2}(\mathcal{M}_{n,k}) = \sum_{i=k}^n { n \choose
  i} { i-1 \choose k-1}, \textrm{ if } k \geq 3.$$
\end{theorem}

The proof of this theorem uses the Goresky-MacPherson theorem so that
the homology can be computed combinatorially.  More specifically, the
Goresky-MacPherson theorem gives the cohomology of the complement of a
subspace arrangement in terms of the homology groups of order complexes
formed from the intersection lattice of the arrangement~\cite{GM}.   For the
space $\M$, the intersection lattice consists of partitions of a
special form and the homology of $\M$ is computed via a detailed
analysis of these partition lattices.

\section{$k$-dimensional spanning trees}

In this section we introduce $d$-dimensional spanning trees for
$d$-dimensional cell complexes.  For any topological space, $X$, we
denote the rank of the $i^{th}$ homology group of $X$ by
$\beta_i(X)$. For any cell complex $\Sigma$, we refer to the cells of
$\Sigma$ as \emph{faces} and write $f_{\ell}(\Sigma)$ for the number of
$\ell$-dimensional faces in $\Sigma$.  The collection of all faces of
dimension $k$ or less is the \emph{$k$-skeleton} of $\Sigma$ and
denoted by $\Sigma_k$.  Finally, any face of maximal dimension is
referred to as a \emph{facet}.

The following definition is not the most general notion of a higher
dimensional tree but is sufficiently general for our purposes and
avoids unnecessary technical complications, see~\cite{chapter} for
more details.

\begin{defin}
Let $\Sigma$ be a $d$-dimensional cell complex such that $\beta_{d-1}(\Sigma) = 0$. A subcomplex $T \subset \Sigma$ such that $T_{d-1} = \Sigma_{d-1}$ is a \emph{$d$-spanning tree} if
  \begin{subequations}
  \begin{align}
  \label{acyc-condn}  & H_d(T, \mathbb{Z}) = 0,\\
  \label{conn-condn}  & |H_{d-1}(T, \mathbb{Z})| < \infty, \quad\text{and}\\
  \label{count-condn} & f_d(T) = f_d(\Sigma) - \beta_d(\Sigma).
  \end{align}
  \end{subequations}
\end{defin}

The initial condition that the $d-1$ skeleta are equal is the spanning
condition.  The other three homological conditions are analogues to
the familiar graphical conditions for a tree on $n$ vertices:
acyclicity, connectedness and having precisely $n-1$ edges.

Spaces which are themselves cellular spanning trees include any
triangulation of a disk, but also any triangulation of ${\R}P^2$.
Condition~\ref{conn-condn} allows for the presence of torsion which leads to much of the interesting structure of trees.  If
$\Sigma$ is the boundary of a convex polytope in $\mathbb{R}^d$, then any collection of all but
one facet gives a $d$-dimensional spanning tree.  More generally,
cellulated spheres are the higher dimensional analogue of cycle graphs (i.e. cellulated one-spheres)
where the removal of any one edge yields a spanning tree.

Here we will be primarily concerned with spanning trees of cubes and their skeleta.  
Let $\Cube_n$ denote the $n$-dimensional hypercube, thought of either as a
geometric convex polytope or a combinatorial cell complex.  As a
geometric object $\Cube_n$ is the convex hull of the $2^n$ points in
$\mathbb{R}^n$ whose coordinates are all $0$ or $1$.  Combinatorially,
 the face lattice consists of all ordered $n$-tuples
$(\sigma_1, \sigma_2, \ldots, \sigma_n)$, where $\sigma_i \in
\{0,1,*\}$.  A face $\sigma$ is contained in a face $\tau$ if
$\sigma_i \leq \tau_i \, \forall i$, where the digits are ordered $0 <
*, 1 < *$, and $0,1$ are incomparable.  With this encoding, the
dimension of a face is simply the number of $*$s in its string.
  Let $\Cube_{n,k}$ denote the $k$-skeleton of the $n$-cube, then
the facets of $\Cube_{n,k}$ are all $\{0,1,*\}$ strings of length $n$ with
exactly $k$ $*$s.  

Let $T \subset \Cube_{n,k}$ be a cellular spanning tree of
$\Cube_{n,k}$.  Hence $T$ contains the entire $k-1$ skeleton
$\Cube_{n,{k-1}}$ and some collection of $k$-dimensional facets of
$\Cube_{n,k}$, see \cite{cubes} for a detailed study of spanning trees of cubical complexes.  The size of $T$, i.e. the number of facets of $T$, or equivalently, the
$k$th entry of the $f$-vector $f_k(T)$ is:

\begin{equation*}
\label{size}
|T| = f_k(T) = \sum_{i=k}^n { n \choose
  i} { i-1 \choose k-1}.
\end{equation*}

The hypercube $\Cube_n$ is dual  to the $n$-dimensional
cross polytope, $\Cross_n$.  Namely, there is an inclusion
reversing bijection from the cells of $\Cube_n$ to the cells of
$\Cross_n$.  Moreover, as algebraic cell complexes, the boundary maps of
$\Cube_n$ equal the coboundary maps of $\Cross_n$.  The cross
polytope is realized as the convex hull of the $n$ standard basis vectors of
$\mathbb{R}^n$ and their opposites:

$$ \textrm{Cross}_n = \{ x \in \mathbb{R}^n :  |x_1| + |x_2| + \ldots + |x_d| \leq 1 \}. $$

The hypercube is a simple polytope, each vertex of $\Cube_n$ is
contained in precisely $n$ facets.  Dually, the crosspolytope
$\Cross_n$ is a simplicial polytope, each facet of $\Cross_n$
contains precisely $n$ vertices.

\section{Simplicial Resolutions}

The last bit of background information that concerns us is an elementary construction of a kind of ``simplicial resolution''. For a finite set of points $S \subset \R^n$, let $\conv(S)$ denote the convex hull of $S$.

We say that a (compact) set $X\subset\R^n$ is \emph{$m$-avoiding}  if for any $2m$-tuple of distinct points $\{x_1,\ldots,x_m, x'_1,\ldots,x'_m\}, x_k,x'_k\in X, 1\leq k\leq m$, the convex hulls $\conv(x_1,,\ldots,x_m), \conv(x'_1,\ldots,x'_m)$ of these tuples do not intersect.

The following Lemma, which is an immediate corollary of the Thom Transversality Theorem (see e.g. \cite[Chapter 3]{Hirsch} or \cite[Chapter 4]{Wall}) shows that any subset $X$ of $\R^n$ can be embedded as an $m$-avoiding subset: 
\begin{lemma}
For any $m$ and large enough $N$, a generic polynomial embedding of $\R^n$ into  $\R^N$ is $m$-avoiding.
\end{lemma}

This will be useful in the following situation which we will encounter later on:

\begin{defin}\label{def:simplres}
  Let $f : X \to Y$ be a continuous surjective map such that $|f^{-1}(y)|\leq m$ for all $y \in Y$.
  Let $i : X \to \R^n$ be an $m$-avoiding embedding. Define $X^{\Delta}$ by:
\begin{align*}
X^{\Delta} = \{ (y, z) \in Y \times \R^n : z \in conv(i(f^{-1}(y))) \}.
\end{align*}
The extension of $f$ to $X^{\Delta}$ is well-defined because of 
the $m$-avoiding condition.  Denote this extension as $f^{\Delta}$.

The \emph{simplicial resolution} of $(f, i)$ is the pair $(X^{\Delta}, f^{\Delta})$.
\end{defin}

Note that if $X$ is a compact subset of $\R^n$, then so is $X^{\Delta}$.
The property of simplicial resolutions that we will be most concerned with is the following:
\begin{proposition}
\label{simp_res}
For a simplicial mapping between simplicial complexes $f:X\to Y$, its simplicial resolution
$$
f^\Delta:X^\Delta\to Y
$$
is a homotopy equivalence.
\end{proposition}
\begin{proof}
  Indeed, in this situation, the mapping is a fibration with contractible fibers.
  \end{proof}

\section{Proof of the main theorem}\label{Proof}

The final observation we will need concerns the relative sizes of
trees across dimension and duality.  First, an Alexander duality for trees.  

\begin{proposition}\cite[Proposition 6.1]{cubes}
Let $X$ and $Y$ be dual $d$-dimensional complexes and $f^*$ be the inclusion reversing bijection from cells of $X$ to cells of $Y$. Furthermore let $T \subseteq X_i$ and $U = \{f^* | f \in X_i \backslash T \}$.  Then $T$ is an $i$-tree of $X$ if and only if $U$ is a $(d-i)$-tree of $Y$.  
  \end{proposition}

Second, spanning trees of a complex $\Sigma$ in adjacent dimensions $\Sigma_i$, $\Sigma_{i+1}$ have complementary size.  This result appears, e.g., as Proposition 2.6 of \cite{cubes}.  There the proof is formulated in terms of the long exact sequence for relative homology.  We give an alternative argument here for polytopes that relates more directly to our proof of the main theorem.

\begin{proposition}
\label{tree_to_beti}
  Let $P$ be a convex polytope in $\R^n$,  $P_k$  its $k$-skeleton and $T$ a $k$-dimensional spanning tree of $P_k$. Then $f_k(T) = \beta_{k-1}(P_{k-1})$.
\end{proposition}
\begin{proof}
  
  By definition, we have
  $$f_k(T) = f_k(P_k) - \beta_k(P_k).$$

  Because $P_k$ is shellable, it is homotopy equivalent to a wedge of spheres. Thus, its Euler characteristics is
\begin{align*}
  \chi(P_k) = 1 + (-1)^{k} \beta_k(P_k).
\end{align*}
We may also express the Euler characteristic as an alternating sum of the numbers of faces in each dimension:
\begin{align*}
\chi(P_k) = \sum_{i = 0}^{k} (-1)^{i} f_i(P_k).  
\end{align*}
Using the same relations for $P_{k-1}$ and the fact that $\chi(P_k) = \chi(P_{k-1}) + (-1)^{k} f_k(P)$, one gets the desired result.

\end{proof}
Specializing to the case of the cube, we conclude that the following are equinumerous:
\begin{itemize}
\item  the size of a $k$-dimensional tree of $\Cube_{n, k}$
\item the size of a $(n-k)$-dimensional tree of $\Cross_{n, n-k}$ 
\item  the size of the complement of a $(k-1)$-dimensional tree of $\Cube_{n, k-1}$
\item  the size of the complement of a $(n-k-1)$-dimensional tree of $\Cross_{n, n-k-1}$
\end{itemize}
  where $\Cross_{n, k}$ denotes the $k$-dimensional skeleton of the $n$-dimensional cross-polytope and the complements are all taken within the appropriate skeletons.  Numerically, this gives:

\begin{eqnarray*}
f_k(T(\Cube_{n,k})) 
& = & f_{n-k}(T(\textrm{Cross}_{n, n-k})) \\
& & \\
& = & {n \choose {k-1}} 2^{n-k+1} - f_{k-1}(T(\Cube_{n,k-1})) \\
& & \\
& = &  {n \choose k+1} 2^{n-k-1} - f_{n-k-1}(T(\textrm{Cross}_{n, n-k-1}).
\end{eqnarray*}

We are now ready to prove our main result, Theorem~\ref{main}.  

\goldbach*

\begin{proof}

  First, assume $k < n$.

  As discussed above, by Alexander duality, we have:
  $$\beta_{k-1}(\Cube_{n, k-1}) = \beta_{n-k-1}(\Cross_{n, n-k-1})$$

  The $(n-k-1)$-skeleton of $\Cross_n$ consists of simplices that are
  convex hulls of $(n-k)$ of its vertices.  These simplices can be defined as follows.   For any $I  = \{i_1, \ldots, i_k \, | \, 1 \leq i_1 < \ldots < i_k \leq n\}$, let $L_I$ denote the subspace:
 $$L_I = \{x_{i_1} = \ldots = x_{i_k} = 0\}.$$  The faces of the $(n-k-1)$ skeleton of $\Cross_n$ are intersections of the $L^1$-sphere with subspaces of the form $L_I$. 
  We will denote the union of all such $L_I$ by $\Coor_k$, the codim-$k$ coordinate arrangement.
Now, consider the suspension of the intersection of the $L^1$-sphere and $\Coor_k$. The suspension is homeomorphic to the one point compactification of $\Coor_k$, ${\Coor_k}^*$. Thus, $\beta_{n-k-1}(\Cross_{n, n-k-1}) = \beta_{n-k}({\Coor_k}^*)$.

Let $S = \{(x_1, \ldots, x_n) \in \R^n \, | \, \sum_{i=0}^n x_i = 0\}$ and
let $$\pi : \Coor_k \rightarrow S$$ be the projection of the coordinate arrangement to $S$ along the diagonal.  
Note that the image $\pi(\Coor_k)$ lands inside $\mathcal{A}_{n, k}$.
Furthermore, this extends continuously to one point compactifications. Slightly abusing notation, we continue to use $\pi$ to refer to this extension.

We are now in the situation of Definition \ref{def:simplres} -- we may safely assume that the one-point compactifications of our arrangements are triangulated subsets of spheres in Euclidean space.

In the case that $n < 2 k$, $\pi$ is a homeomorphism. However, when $n \geq 2 k$, it is not: the point where several $k$-diagonals intersect has multiple preimages. The number of preimages is bounded from above by $m=\lfloor n/k\rfloor$.

Consider the simplicial resolution of $(\pi, i)$, $(\Coor_k^*)^\Delta$. Using Theorem \ref{simp_res}, $\pi$ is a homotopy equivalence. Thus, $\beta_{n-k}((\Coor_k^*)^{\Delta}) = \beta_{n-k} (A_{n, k}^*)$.
All simplices added while taking the simplicial resolution are of dimension at most $n-k-2$: indeed, the dimension of the cells glued over the preimages of $l$-fold intersections of the $k$-diagonals is equal to
$$
n-l(k-1)+(l-1)
$$
(the first summand is the dimension of the $l$-fold intersection; the second, of the simplices over each point of the self-intersection). As $l\geq 2$ and $k\geq 3$, we obtain the desired bound.

Therefore, the cells added to $\Coor_k^*$ to obtain the simplicial resolution do not affect homology in dimension $n-k$. Therefore, $\beta_{n-k}(\Coor_k^*) = \beta_{n-k} (\mathcal{A}_{n, k}^*)$.
Finally, by Alexander duality, $\beta_{n-k} (\mathcal{A}_{n, k}^*) = \beta_{k-2} (\M)$ and  $f_k(T) = \beta_{k-2} (\M)$ as desired.

For $k = n$, the $n$-dimensional hypercube is an $n$-dimensional spanning tree of itself; $f_n(T) = 1$. The $n^{th}$ no $n$-equal space of $\R$ is homotopy equivalent to an $(n-2)$-dimensional sphere, so $\beta_{n-2} (\mathcal{M}_{n,n}) = 1$. Thus, the claim holds for all $k \leq n$.

\end{proof}

\section{Piles of Cubes} \label{piles}
The identity in Theorem \ref{main} can be generalized to the following situation. Consider the { comb} no-$k$-equal subspace arrangement defined as follows:
\begin{defin}
  Let $A_j\subset \R , j=1,\ldots,n$ be finite non-empty subsets of the reals.
  
  The \emph{$A$-comb $k$-equal arrangement} of $\R^n$ consists of all subspaces of the form $\{x_{i_1} - a_{i_1} = \cdots = x_{i_k} - a_{i_k}\}$ for $1 \leq i_1 < \ldots < i_k \leq n$ and $a_{i_j} \in A_{i_j}$.
  
  The \emph{$A$-comb no $k$-equal space} of $\R$ is the complement in $\R^n$ of the $A$-comb $k$-equal arrangement.

  We will denote this aforementioned arrangement as $\Delta^A_k\subset\R^{n-1}$, and its complement as $M^A_k$.

\end{defin}

  Notice that we recover the no $k$-equal arrangement when all the $A_j$s are $\{ 0 \}$.

  Define a {\em $k$-dependence} between the sets $A_j$ as a collection of $k$ distinct pairs $\{x_{j_1},x'_{j_1}\}\in A_{j_1}, \ldots, \{x_{j_k},x'_{j_k}\}\in A_{j_k}$
  such that $x_{j_i}-x'_{j_i}$ coincide for all $i=1,\ldots,k$. 

  \begin{defin}
 A {\em pile of cubes of size $\prod_{j=1}^n N_j$} is the (cubical) CW complex consisting of the parallelogram $[0,N_1]\times [0,N_2]\times\cdots\times [0,N_n]$ naturally stratified by the integer grid.
  \end{defin}
  
Theorem~\ref{general} can now be written more precisely as:

\begin{theorem}
\label{gen}
  Assuming that the there are no $k$-dependences between the $A_j$s, the rank of the $(k-2)-dimensional$ homology of $M^A_k$
  is equal to the number of facets in a $k$-dimensional spanning tree of the $k$-skeleton of the pile of cubes of size $\prod_{j=1}^n n_j$.
\end{theorem}

The key component of the proof is the following result:

\begin{proposition}
  The rank of the $(n-k)$-th integer homology of the one-point compactification of the
  arrangement $\Delta^A_{k}$ equals the rank of the $(k-1)$-st integer homology of the $(k-1)$-st skeleton of the
  pile of cubes of size $\prod_{j=1}^n n_j$.
\end{proposition}
\begin{proof}
  We start with a natural construction of a pile of cubes in $\R^n$: pick one point in the interior of the $n_j$ open intervals
  into which $A_j$ partitions $\R$. We will denote this subset as $B_j$. The product of the collections of the $n_j$ closed
  intervals in the $j$-th factor of $\R^n$ defines a pile of cubes $B$ of size $\prod_{j=1}^n n_j$.

  We consider our Euclidean
  $n$-space $\R^n\subset S^n$ as an open subset of its one-point compactification. Adding the large open cell at infinity to
  the pile of cubes $B$ defines a (cubical) regular CW complex structure on the $n$-sphere.

  On the other hand, we have a natural CW complex obtained by taking the products of the points of the $A_j$s
  and the intervals into which $A_j$s split the real line. This CW complex can be compactified into a
  finite regular CW complex by adding a point at infinity; we will denote this complex as $A$.
  Both $A$ and $B$ are homeomorphic to the $n$-sphere.

  Importantly, these two CW complexes are dual: for each $k$ cell of one there exists exactly one $(n-k)$ cell
  of the other, intersecting  at a unique point, and the boundary operators on these two complexes are automatically dual to each other.

  This implies that the $k$-th homology of the $k$-skeleton of one of these CW-complexes is isomorphic to the $(n-k-1)$-st homology of the $(n-k-1)$-skeleton of the other. Thus, the $(k-1)$-st homology of the $(k-1)$-skeleton of $B$ is isomorphic to the $(n-k)$-th homology of the $(n-k)$-skeleton of $A$.
  
  Analogous to the proof of Theorem~\ref{main}, we consider the projection of $A$ into $S = \{(x_1, \ldots, x_n) \in \R^n \, | \, \sum_{i=0}^n x_i = 0\}$. The image of this projection lives in $\Delta^A_{k}$. We may once again extend this to a one point compactification. Once more consider the simplicial resolution of this projection. The fact that there are no $k$-dependences between the $A_j$s ensures that the dimension of the cells added in the construction of the simplicial resolution are at most $n-k-2$. Thus, the $(n-k)$-th homology of the $(n-k)$-skeleton of $A$ is isomorphic to the $(n-k)$-th homology of the one point compactification of $\Delta^A_{k}$.
\end{proof}

The rest of the proof of Theorem~\ref{gen} follows from Proposition~\ref{tree_to_beti} at the beginning and Alexander duality at the end.

\begin{cor}
Assuming that the there are no $k$-dependences between the $A_j$s, the rank of the $(k-2)-dimensional$ homology of $M^A_k$, $\beta_{k-2}$, satisfies the following:
\[1+(-1)^{k-1} \beta_{k-2} = \prod_{j=1}^n{(n_j+1)} \left(\sum_{\ell=0}^{k-1} (-1)^\ell \sum_{|I|=\ell}\prod_{i\in I}\frac{n_i}{n_i+1}   \right)\]
where the $I$ are subsets of $\{1, \dots, n\}$.
\end{cor}

\begin{proof}
By Theorem~\ref{gen}, $\beta_{k-2}$ equals the number of facets in a $k$-dimensional spanning tree of the $k$-skeleton of the pile of cubes of size $\prod_{j=1}^n n_j$. Let $P_{\ell}$ denote the $\ell$-skeleton of this pile of cubes.
By Proposition~\ref{tree_to_beti}, the number of facets in a $k$-dimensional spanning tree of $P_k$ is equal to $\beta_{k-1}(P_{k-1})$. $\beta_{k-1}(P_{k-1})$ satisfies
\[1+(-1)^{k-1} \beta_{k-1}(P_{k-1}) = \prod_{j=1}^n{(n_j+1)} \left(\sum_{\ell=0}^{k-1} (-1)^\ell \sum_{|I|=\ell}\prod_{i\in I}\frac{n_i}{n_i+1}   \right).\]
The left hand side is the Euler characteristic of $P_{k-1}$ computed using the fact that $P_{k-1}$ is homotopy equivalent to a wedge of spheres. The right hand side is the Euler characteristic computed as an alternating sum of the number of cells in each dimension.
\end{proof}

\bibliographystyle{amsalpha}
\bibliography{biblio.bib}

\end{document}